\documentclass[a4paper,reqno,10pt]{amsart}
\usepackage{amsthm,amssymb,amscd,amsfonts,amsmath,hyperref,amsaddr}

\newtheorem{prop}{Proposition}
\newtheorem{lem}{Lemma}

\newtheorem{rem}{Remark}
\newtheorem{coro}{Corollary}

\numberwithin{equation}{section}
\begin{document}
\title{Global nilpotent cone is isotropic: parahoric torsors on curves}
\author{Rohith Varma}
\address{Tata Institute of Fundamental research, Homi Bhabha Road, Colaba,
	Mumbai-400005}
\email{rvarma.kvm@gmail.com}
\maketitle
\begin{abstract}
	In this note we show that the global nilpotent cone is an isotropic
	substack of the cotangent bundle of the moduli stack of parahoric
	torsors on a smooth projective curve. Further for the case of
	symplectic parabolic bundles, we show that the global nilpotent cone
	is infact a Lagrangian complete intersection substack.
\end{abstract}
\section{introduction}
Let $X$ be a smooth projective curve over $\mathbb{C}$ of genus $g >1$ and
$G$ a complex semisimple group. Denote by $Bun_G$, the
moduli stack of principal $G$-bundles on $X$ and by $\mathbb{P}$, the universal
bundle in $X \times Bun_G$. Let $T^*Bun_G = Spec(Sym^*(ad(\mathbb{P})))$
be the cotangent stack of $Bun_G$. For a scheme $S$ and a principal bundle
$E$ on $X \times S$, it is easy to see that [see \cite{G}]
\[
T^*_E Bun_G = H^0(X \times S, ad(E) \otimes \Omega^1_{X \times S/S})
\]
Now for a point $y \in X \times S$, and $s \in H^0(X \times S, ad(E) \otimes \Omega^1_{X \times S/S})$, we have
$s(y) \in ad(E)(y) \otimes \Omega^1_{X \times S/S}(y) \simeq \mathfrak{g} \otimes k(y)$. Thus it makes sense to call $s$ nilpotent if for
any point $y \in X \times Y$, we have $s(y)$ is a nilpotent element in
$\mathfrak{g} \otimes k(y)$. Following Laumon(\cite{L}) and Ginzburg(\cite{G}), the 
\textit{global nilpotent cone} of $T^*Bun_G$ is defined as the substack
$\mathcal{N}ilp$, whose $S$-valued points are
\[
\mathcal{N}ilp(S) = \{(P,s) \mid P \in Bun_G(S), s \in H^0(X \times S, ad(P) \otimes \Omega^1_{X \times S/S}) \,\ \text{is nilpotent} \}.
\]

An equivalent definition of the nilpotent substack $\mathcal{N}ilp$, is using the Hitchin morphism. Let $\{F_1,\ldots,F_l\}$ be a minimal set of
homogeneous generators for the algebra $\mathbb{C}[\mathfrak{g}]^G$.
Let $d_i$ be the degree of $F_i$ and denote by $H$, the affine space
\[
H = \bigoplus_{i=1}^l H^{0}(X,K_X^{d_i})
\]
Then $dim(H) = dim(Bun_G)$ and Hitchin defines a morphism
\[
\chi : T^*Bun_G \rightarrow H
\]
where for a  scheme $S$, $P \in Bun_G(S)$ and $s \in H^0(X \times S, ad(P) \otimes \Omega^1_{X \times S/S})$, $\chi(s) = (F_1(s),\ldots,F_l(s))$.
Then it follows that $\mathcal{N}ilp$ is the fiber of $\chi$ over $0 \in H$.\\

The \textit{global nilpotent cone} $\mathcal{N}ilp$, has been shown
to be a \textit{Lagrangian} substack of $T^*Bun_G$ by various authors(see \cite{G}, \cite{BD}). As explained in \cite {BD}, it follows from this
fact that the Hitchin morphism $\chi$ is flat and surjective and
further $T^*Bun_G$ is good in the sense of Beilinson and Drinfeld \cite{BD}.\\

We will study in this note, a natural notion of the \textit{global nilpotent
cone} in the case of the cotangent	stack of the moduli stack of parahoric torsors as defined in \cite{BS}. The more familiar notion of parabolic
bundles on curves, is a special case of parahoric torsors. 
The main results we derive regarding the \textit{global nilpotent cone}
for parahoric torsors are the following:\\
(A) The \textit{global nilpotent cone} for parahoric torsors is an
isotropic substack of the cotangent stack of the moduli stack of parahoric
torsors.\\ (see Section~\ref{sec:3},Proposition~\ref{prop:1}).\\
(B) We consider a natural analogue of the Hitchin morphism in the
case of parahoric torsors and conclude from Propositon~\ref{prop:1}, that
if the dimension of the image of the Hitchin morphism is less than or equal to that
of the dimension of the moduli stack of torsors, then the 
\textit{global nilpotent cone} is infact Lagrangian.\\
(C) We prove in the case of the symplectic group $Sp_{2n}$, that the
Hitchin morphism factors through a subvariety of dimension the same as that
of the moduli stack and hence the \textit{global nilpotent cone} is Lagrangian.\\

The case of parabolic bundles with full flags, has been previously studied by 
Faltings \cite{F} and the \textit{global nilpotent cone} was proved
to be Lagrangian. The case of $G = GL_n,SL_n$, has been studied as well
by Peter Scheinost and Martin Schottenloher in \cite{meta}. This work 
is mainly a result of trying to understand the results in \cite{meta}.\\

\section{parabolic bundles}
Let us assume from now on that $G$ is simple and simply-connected.\\
Recall(see \cite{TW})
a parabolic $G$- bundle on $(X,x)$ of weight $\theta$, is a triple $(E,\phi,\theta)$, where\\
(i)E is a principal $G$-bundle.\\
(ii)$\theta \in \mathfrak{U}^o := \{\theta \in Y(T) \otimes \mathbb{Q} \mid (\theta,\alpha_i) \geq 0, \,\ (\theta,\alpha_0) <1  \}$\\
(iii) $\phi$ is a reduction of structure group of $E$ at $x$ to the
parabolic subgroup $P_{\theta} \subset G$, determined by $\theta$.\\

It was shown by Balaji and Seshadri (\cite{BS}), that one can equivalently think of parabolic $G$-bundles on $(X,x)$, as torsors under certain group schemes which are called Bruhat-Tits group schemes. Let us
briefly recall the results from \cite{BS}:\\
Let $\theta \in \mathfrak{U}$ and
for any root $r \in R$, consider the integer
\[
m_r(\theta) = -\lfloor (\theta,r) \rfloor
\]
where $(,) : Y(T) \times X(T) \rightarrow \mathbb{Z}$, is the usual pairing.

Identify the completed local ring $A = \widehat{\mathcal{O}}_{X,x}$, with
the power series ring $\mathbb{C}[[z]]$ .
We denote by $K$, the fraction field of $A$.
Then the subgroup
\[
\mathcal{P}_{\theta} := <T(A), z^{m_r(\theta)}U_r(A)> \subset G(K) 
\]
is a parahoric subgroup in the sense of Bruhat-Tits. Thus from
Bruhat-Tits theory, we get
a smooth affine group scheme $\mathcal{G}_{\theta}$ over $spec(A)$,
which satisfies:\\
(i) $\mathcal{G}_{\theta} \times_{spec(A)} Spec(K) \cong G \times spec(K)$.\\
(ii)$\mathcal{G}_{\theta}(A) = \mathcal{P}_{\theta}$.\\
One can construct a smooth
affine group scheme $\mathcal{G}_{X,x,\theta}$ on $X$ such that:\\
(i)$\mathcal{G}_{X,x,\theta}\mid_{X-x} \cong G \times (X-x)$.\\
(ii)$\mathcal{G}_{X,x,\theta}\mid_{spec(A)} \cong \mathcal{G}_{\theta}$.
In \cite{BS}, it is shown that for $\theta \in \mathfrak{U}^o$, we have
\[
\mathcal{P}_{\theta} = ev^{-1}(P_{\theta}).
\]
where $ev : G(A) \rightarrow G$ is the natural map and $P_{\theta}$ is 
the parabolic subgroup determined by $\theta$. Further in this case,
a $\mathcal{G}_{X,x,\theta}$ torsor on $(X,x)$ is the same as a parabolic
$G$-bundle of weight $\theta$.\\

For the rest of this article, a parabolic $G$-bundle of weight $\theta$
on $(X,x)$ is a $\mathcal{G}_{X,x,\theta}$ torsor on $(X,x)$.
We will denote by $Bun_{\mathcal{G}_{X,x,\theta}}$, the moduli stack
of parabolic $G$ bundles of weight $\theta$ on $(X,x)$.

\section{Global nilpotent cone of $T^*Bun_{\mathcal{G}_{X,x,\theta}}$}\label{sec:3}
We have $Bun_{\mathcal{G}_{X,x,\theta}}$ is a smooth equidimensional algebraic stack. Let $S$ be a scheme over $spec(\mathbb{C})$.
From standard deformation theory, we have
\[
T^*Bun_{\mathcal{G}_{X,x,\theta}}(S) = \{(E,s) \mid E \in Bun_{\mathcal{G}_{X,x,\theta}}(S), s \in H^{0}(X \times S, ad(E)^*\otimes \Omega^1_{X\times S/S}) \}.
\]
Since $\mathcal{G}_{X,x,\theta}\mid_{X-x} \cong G \times (x-x)$,
we have, via the isomorphism $\mathfrak{g} \cong \mathfrak{g}^{\vee}$
induced by the killing form,
$ad(E)^* \otimes \Omega^1_{X \times S/S} \mid_{X -x \times S} \cong ad(E)
\otimes \Omega^1_{X\times S/S} \mid_{X -x \times S}$. Thus for
any point $y \in X-x \times S$, we have
\[
s(y) \in \mathfrak{g} \otimes k(y).
\]
Thus we say $s\mid_{X-x \times S}$ is nilpotent if 
\[
\forall y \in X-x \times S, \,\ s(y) \in \mathfrak{g} \otimes k(y) \,\ \text{is nilpotent}.
\]
We define the \textit{global nilpotent cone} as the substack
$\mathcal{N}ilp_{X,x,\theta}$, defined as 
\[
\mathcal{N}ilp_{X,x,\theta}(S) = \{ (E,s) \in T^*Bun_{\mathcal{G}_{X,x,\theta}} \mid s\mid_{X-x \times S} \text{is nilpotent} \}.
\]
We then have the following 
\begin{prop}\label{prop:1}
	$\mathcal{N}ilp_{X,x,\theta}$ is an isotropic substack of $T^* Bun_{\mathcal{G}_{X,x,\theta}}$.
\end{prop}

\begin{proof}
	The proof is an adaptation of \cite[Lemma~5,pg~516]{G}and hence
	we will contend ourselves by explaining the necessary modifications involved. Let $E$ be a $\mathcal{G}_{X,x,\theta}$ torsor on $(X,x)$
	and $s \in H^{0}(X,ad(E)^* \otimes K_X)$ be nilpotent. Let $\mathcal{B} \subset \mathcal{G}_{X,x,\theta}$ be the borel subgroup scheme,
	defined as the flat closure of $X-x \times B$ in
	$\mathcal{G}_{X,x,\theta}$, for a borel subgroup $B \subset G$.
	From Heinloth~\cite{H}[Lemma~23], we have the natural morphism
	$f: Bun_{\mathcal{B}} \rightarrow Bun_{\mathcal{G}_{X,x,\theta}}$ is surjective.\\
	Now we can  find a finite Galois cover $p : Y \rightarrow X$,
	with Galois group $\Gamma$, such that the stack of $\Gamma$ equivariant principal
	$G$-bundles on $Y$ of a fixed local type determined by $\theta$, 
	is equivalent to the stack of $\mathcal{G}_{X,x,\theta}$ torsors on
	$X$. For a choice of such an equivariant bundle $F$ on $Y$, we have
	\[
	\mathfrak{R}_{Y/X}(Ad(F))^{\Gamma} \cong \mathcal{G}_{X,x,\theta}
	\]
	where $\mathfrak{R}_{Y/X}()$ denotes the restrcition of scalars functor.
	The equivalence of the categories of equivariant bundles on $Y$
	and $\mathcal{G}_{X,x,\theta}$ is obtained by
	\[
	N \mapsto \mathfrak{R}_{Y/X}^{\Gamma}(N \wedge F^{op})
	\]
	Further we have a $\Gamma$ equivariant $B$ reduction $F_B$ of $F$, such
	that $\mathfrak{R}_{Y/X}^{\Gamma}(Ad(F_B)) \cong \mathcal{B}$.
	As in \cite{G}, choose a $B$ reduction of $F$ over the generic point
	of $X$, so that we have $s \in \mathfrak{n}_F \otimes K_X$.
	If $N$ is the equivariant bundle on $Y$, which corresponds to $F$
	under the equivalence mentioned above, then we have an equivariant section 
	$\tilde{s}$ of $ad(N) \otimes K_Y$, which descends to $s$. 
	Further we get an equivariant $B$ reduction of $N$ over the generic point of $Y$, so that
	$\tilde{s}$ is a section of $\mathfrak{n}_N \otimes K_Y$.
    Since $G/B$ is projective, we can extend this $B$ reduction $N_B$ of $N$
    to the whole of $Y$. Thus we have a $\mathcal{B}$ reduction of
    $F$ given by 
    \[
    F_{\mathcal{B}} = \mathfrak{R}_{Y/X}^{\Gamma}(N_B \wedge E_B)
    \] 
    such that
    over the generic point, $s \in \mathfrak{n}_F \otimes K_X$.
    Thus we have $f^*(s) \in H^{0}(X,ad(F_{\mathcal{B}}) \otimes K_X) = 0$
    as it vanishes on the generic point.   
    Following the notations of \cite[Lemma 3]{G},  for $N_1 = Bun_{\mathcal{B}}$ and $N_2 = Bun_{\mathcal{G}_{X,x,\theta}}$, we have
    $\mathcal{N}ilp_{X,x,\theta} = pr_2(Y_f)$. The rest of the proof
    can be done exactly as in the proof of \cite{G}[Lemma~5].
\end{proof}

\begin{rem}
	In the above Proposition, we have nowhere used the fact that $\theta$ lies
	in the interior of the weyl alcove. Hence the above Lemma holds true
	for general parahoric torsors as considered in \cite{BS}.
\end{rem}
\section{Hitchin map for parahoric torsors}
Let $\widehat{\mathfrak{p}} \subset \mathfrak{g}((z))$ be a parahoric subalgebra.
Let $\kappa(,)$ be the killing form on $\mathfrak{g}$.
We define the \textit{dual}  of $\widehat{\mathfrak{p}}$ as
\[
\widehat{\mathfrak{p}}^{\vee} = \{ u \in \mathfrak{g}((z)) \mid \kappa(u,v) \in \mathbb{C}[[t]], \,\ \forall v \in \mathfrak{p} \}
\] 
We have a natural isomorphism
\[
\Psi : \widehat{\mathfrak{p}}^{\vee} \rightarrow Hom_{\mathbb{C}[[t]]}(\widehat{\mathfrak{p}},\mathbb{C}[[t]])
\]
given by
\[
\Psi(u):= \Psi(u)(v) = \kappa(u,v), \,\ \forall v \in \widehat{\mathfrak{p}}, \,\ u \in \widehat{\mathfrak{p}}^{\vee}.
\]
Hence it makes sense to call $\widehat{\mathfrak{p}}^{\vee}$, the \textit{dual}
of $\widehat{\mathfrak{p}}$.\\

Consider now the case when $\widehat{\mathfrak{p}} = ev^{-1}(\mathfrak{p})$ for a parabolic sub-algebra
$\mathfrak{p} \subset \mathfrak{g}$. We then have the following lemma
\begin{lem}\label{lem:1}
	Let $\widehat{\mathfrak{p}} = ev^{-1}(\mathfrak{p})$ for a parabolic
	subalgebra $\mathfrak{p} \subset \mathfrak{g}$.
	We then have a natural equality of lattices in $\mathfrak{g}((z))$
	\[
	\frac{1}{z}(\mathfrak{n} + z\mathfrak{g}[[z]]) = \widehat{\mathfrak{p}}^{\vee}.
	\]
	where $\mathfrak{n}$ is the nil-radical of $\mathfrak{p}$.
\end{lem}

\begin{proof}
Let $\mathfrak{t} \subset \mathfrak{p}$ be a cartan subalgebra.
Let $R = R^{+} \cup R^{-}$ be the set of roots with respect to $\mathfrak{t}$.
We then have the familiar decomposition
\[
\mathfrak{g} = \mathfrak{t} \oplus \bigoplus_{\beta \in R} \mathfrak{g}_{\beta}
\]
Let $\{x_{\beta} \in \mathfrak{g}_{\beta}\}_{\beta \in R} \cup \{h_1,\cdots,h_l \mid h_i \in \mathfrak{t}\}$ be a basis of
$\mathfrak{g}$.
Using the killing form $\kappa$, we can identify $\mathfrak{g}$ and
$\mathfrak{g}^{\vee}$. Consider now
the linear function $f_{\beta}$ for $\beta \in R$, given by
\[
f_{\beta}(g) = \begin{cases}
0 & g \in \mathfrak{t}\\
0 & g \in \mathfrak{g}_{\alpha}, \alpha \neq \beta.\\
c & g = cx_{\beta}.
\end{cases}
\]
We then have under the isomorphism $\mathfrak{g} \cong \mathfrak{g}^{\vee}$,
via the killing form, the linear function $f_{\beta}$ corresponds to $x_{-\beta}$.
Similarly we can define $f_i$ by
\[
f_{i}(g) = \begin{cases}
0 & g \notin \mathfrak{t}\\
0 & g \in \mathfrak{t} ,\,\ g \in <h_{j}>_{j\neq i}\\
c & g = ch_i.
\end{cases}
\]
We then have $f_i$ corresponds to an element in $\mathfrak{t}$ under
the isomorphism $\mathfrak{g} \cong \mathfrak{g}^{\vee}$.
Now let $\mathfrak{p} = \mathfrak{t} \oplus \bigoplus_{\beta \in R^{+}}\mathfrak{g}_{\beta} \oplus \bigoplus_{\alpha \in R^* \subset R^{+}} \mathfrak{g}_{-\alpha}$ and $\mathfrak{n} = \bigoplus_{\delta \in R^*}\mathfrak{g}_{\delta}$.\\
We then have
\[
\widehat{\mathfrak{p}} = \mathfrak{t}[[z]] \oplus \bigoplus_{\beta \in R^{+}}\mathfrak{g}_{\beta} \otimes \mathbb{C}[[z]] \oplus \bigoplus_{\alpha \in R^* \subset R^{+}} \mathfrak{g}_{-\alpha} \otimes \mathbb{C}[[z]] \oplus
\bigoplus_{\delta \notin R^*} \mathfrak{g}_{-\delta} \otimes z\mathbb{C}[[z]]
\]
Hence 
\[
\widehat{\mathfrak{p}}^{\vee} = <f_i,\{f_{\beta}\}_{\beta \in R^{+}}, \{f_{-\alpha}\}_{\alpha \in R^*}, \{\frac{1}{z}f_{-\delta}\}_{\delta \notin R^{*}} \} >\mathbb{C}[[z]].
\]
Thus we get from the above discussion,
\begin{eqnarray*}
\widehat{\mathfrak{p}}^{\vee} & = &  <\mathfrak{t},\{\mathfrak{g}_{\beta}\}_{\beta \notin R^*},\{\frac{1}{z}\mathfrak{g}_{\delta}\}_{\delta \in R^*}>\mathbb{C}[[z]]\\
                              & = & \frac{1}{z}(\mathfrak{n} + z\mathfrak{g}[[z]])
\end{eqnarray*}
\end{proof}

Let $E$ be a $\mathcal{G}_{X,x,\theta}$ torsor on $X$. Let 
$\mathfrak{p} = Lie(\mathcal{P}_{\theta}) \subset \mathfrak{g}((z))$,
which is a parahoric subalgebra of $\mathfrak{g}((z))$. We then have
an idenitification of lattices in $\mathfrak{g}((z))$
\[
\widehat{ad(E)}^{*}_x = \mathfrak{p}^{\vee}
\]
where $\widehat{ad(E)}^*_x$ is the completion of the stalk of $ad(E)^*$ at $x$.
Now choose a homogeneous generating set $\{F_1,\ldots,F_l\}$ for $\mathbb{C}[\mathfrak{g}]^G$. Since the set $\widehat{p}^{\vee}$
is a bounded subset of $\mathfrak{g}((t))$, we have for $C>>0 \in \mathbb{N}$,
\[
\nu_t(F_i(\alpha)) \geq -C , \,\ \forall 1 \leq i \leq l, \,\ \forall \alpha \in \mathfrak{p}^{\vee}.
\]
Thus the we can define a Hitchin morphism
\[
\chi : Bun_{\mathcal{G}_{X,x,\theta}} \rightarrow \bigoplus_{i=1}^l H^{0}(X,K_X^{\otimes d_i}(Cx))
\]
An immediate corollary of Proposition~\ref{prop:1} is the following:
\begin{coro}\label{coro:1}
If we can find a subvariety $B \subset \bigoplus_{i=1}^l H^{0}(X,K_X^{\otimes d_i}(Cx))$, such that 
$dim(B) = dim(Bun_{\mathcal{G}_{X,x,\theta}})$
and	$\chi$ factors through $B$, then we have \\
(i)$dim(\mathcal{N}ilp_{X,x,\theta}) = dim(Bun_{\mathcal{G}_{X,x,\theta}})$\\
(ii) $\chi$ is flat and surjective.\\	
(iii)$\mathcal{N}ilp_{X,x,\theta}$ is a Lagrangian substack of $T^*Bun_{\mathcal{G}_{X,x,\theta}}$.
\end{coro}

\begin{proof}[sketch of the proof]
	We only give a sketch of the proof, since we just have to follow the arguments  
	in \cite{G}[Proposition~1,Corollary~9, Theorem~10]. 
	We have $dim(T^*(Bun_{\mathcal{G}_{X,x,\theta}})) \geq 2 dim(Bun_{\mathcal{G}_{X,x,\theta}})$ as $Bun_{\mathcal{G}_{X,x,\theta}}$
	is an equidimensional stack (see \cite{BD}). The fact that
	$dim(B) = dim(Bun_{\mathcal{G}_{X,x,\theta}})$, tell us that
	any fiber of $\pi$ has dimension greater than or equal to 
	$dim(T^*(Bun_{\mathcal{G}_{X,x,\theta}})) - dim(B) \geq dim(Bun_{\mathcal{G}_{X,x,\theta}})$.
	But as we have from Proposition~\ref{prop:1},
	that $\mathcal{N}ilp_{X,x,\theta}$ is isotropic, we must have
	$dim(\mathcal{N}ilp_{X,x,\theta}) = dim(Bun_{\mathcal{G}_{X,x,\theta}})$
	and $\mathcal{N}ilp_{X,x,\theta}$
	is infact Lagrangian. We have $\mathcal{N}ilp_{X,x,\theta}$ is the
	fiber over $0$ of the Hitchin morphism. Using The natural $\mathbb{C}^{\times}$ action on the cotangent
	stack, we can put any fiber of $\chi$ in a family parametrized by
	$\mathbb{A}^1$, so that the central fiber is a substack of 
	$\mathcal{N}ilp_{X,x,\theta}$ and all other fibers have the same dimension.
	Thus we must have the dimension of the general fiber is less
	than that of the special fiber, which equals $dim(Bun_{\mathcal{G}_{X,x,\theta}})$. Hence 
	we get
	\[
	dim(\chi^{-1}(b)) = dim(Bun_{\mathcal{G}_{X,x,\theta}}), \,\ \forall b \in B.
	\]
	Hence $\chi$ is flat and surjective.

\end{proof}

\begin{coro}\label{coro:2}[case of parabolic bundles with full flag]
	For $\theta \in \mathfrak{U}^0$, such that $P_{\theta}$ is a Borel subgroupof $G$, we have the global nilpotent cone $\mathcal{N}ilp_{X,x,\theta}$ is Lagranigian.
\end{coro}

\begin{proof}
From Corollary~\ref{coro:1}, it is enough to show that the Hitchin morphism
factors through a subvariety of dimension equal to that of $Bun_{\mathcal{G}_{X,x,\theta}}$. Now we have
\[
dim(Bun_{\mathcal{G}_{X,x,\theta}}) = dim(G)(g-1) + dim(G/B).
\]
where $B = P_{\theta}$.
On the otherhand, let $\widehat{\mathfrak{b}} = ev^{-1}(Lie(B))$.
Let $\mathfrak{n}$ be the nil-radical of $Lie(B)$. 
Let $E$ be a$\mathcal{G}_{X,x,\theta}$ torsors on $(X,x)$.
We have
from Lemma~\ref{lem:1},  and the above discussions
\[
\widehat{ad(E)}_x^{*} \simeq \widehat{\mathfrak{p}}^{\vee} = \frac {1}{z}(\mathfrak{n} + z\mathfrak{g}[[z]]).
\]
Let $\{f_1,\ldots,f_l\}$ be homogeneous invariants generators 
for $\mathbb{C}[\mathfrak{g}]^G$ and $deg(f_i) = d_i$.
Recall, we have 
\[
\Sigma_i d_i = dim(G/B) + l.
\]
For any $e \in \mathfrak{n}$, we have
\[
\nu_z(f_i(e + M)) \geq 1, \,\ \forall M \in z\mathfrak{g}[[z]], \forall j
\]
as $f_j(e) = 0$.
Thus we get
\[
\nu_z(f_i(\frac{1}{z}(e+ M))) \geq -(d_i -1), \,\ \forall i, \forall M \in z\mathfrak{g}[[z]]
\]
Thus we have the image of the Hitchin morphism is contained in the
subspace
\[
W = \bigoplus_i H^{0}(X,K_X^{d_i}(d_i-1))
\]
From Riemann-Roch theorem, we get
\[
dim(W) = dim(G)(g-1) + \Sigma_i (d_i-1) = dim(G)(g-1) + dim (G/B)
\]
\end{proof}

\section{some local computations}
In this section, we will study the local picture of the  Hitchin morphism in the case
of $\mathfrak{g}l_m$, for a particular choice of invariant genartors,
which are the various co-efficients of the characteristic polynomial.
The local study helps us conclude that, for a suitable choice of
invariant generators for $\mathfrak{s}p_{2n}$, the Hitchin morphism
for symplectic parabolic bundles, factors through a subvariety of
dimension same as that of the moduli stack. Thus from Corollary~\ref{coro:1},
we get the \textit{global nilpotent cone} is Lagrangian.
The combinatorial results and the other statements we derive in this
section, are a result of trying to understand the statements found
in \cite{meta}[pages~209-215]. \\

Let $e \in \mathfrak{g}l_m$ be a nilpotent element. We can
associate to the conjugacy class of $e$, a partition 
$\mu = (m_1,\ldots,m_k)$ of $m$. We denote by $\tilde{\mu} = (\tilde{m}_1,\ldots,\tilde{m}_r)$, the dual partition of $m$.
Now assume $m=2n$ and $e \in \mathfrak{s}p_{2n} \subset \mathfrak{g}l_{2n}$. 
It is a well-known fact that we have, the corresponding partition $\mu$, satisfies the property

\[
\text{in $\mu$ every odd number occurs with even multiplicity }
\]
Further we have
\[
dim(Z_{SP_{2n}}(e)) = \frac{1}{2}(\Sigma_i \tilde{m}_i^2 + \#\{i \mid m_i \,\ \text{is odd} \})
\]
Now for a number $1 \leq j \leq m$, we define
\[
n(\mu,j) = a; \,\ \text{where $a$ is the unique number such that} \,\ m_1 + \cdots + m_{a-1} < j \leq
m_1 + \cdots + m_a.
\]
We denote by $F_j \in \mathbb{C}[\mathfrak{g}l_m]^{GL_m}$, defined by
\[
det(Q - TId_{m\times m}) = T^m + F_1(Q)T^{m-1} + \ldots + F_m(Q).
\]
Following Khazdan and Lusztig (\cite{KL}), we say a subset $Y \subset \mathfrak{g}[[z]]
$ is constructible, if for some integer $N$, we have
a constructible subset $Y_l$ of the affine space $\mathfrak{g}[[t]]/t^l \mathfrak{g}[[t]]$, such that under the natural map
\[
f_l : \mathfrak{g}[[t]] \rightarrow \mathfrak{g}[[t]]/t^l\mathfrak{g}[[t]]
\]
we have
\[
Y = f_l^{-1}(Y_l).
\]
A subset $U \subset Y$, where $Y$ is constructible as above is called
open, if for some $l_1 > l$, we have $ \tilde{U} \subset (p^{l_1}_l)(Y_l)$
open and $U = p_{l_1}^{-1}(\tilde{U})$, where
\[
p^{l_1}_l : \mathfrak{g}[[t]]/t^{l_1}\mathfrak{g}[[t]] \rightarrow
\mathfrak{g}[[t]]/t^l\mathfrak{g}[[t]].
\]
is the natural map.
Now we claim the following
\begin{lem}\label{lem:2}
	Consider the subset
	$\mathcal{I}^j_e \subset e + z\mathfrak{g}l_m[[z]]$, given by
	\[
	\mathcal{I}^j_e = \{\gamma \in e + z\mathfrak{g}l_m[[z]] \mid \nu_t(\gamma)
	\,\ \text{is the minimum} \}
	\] 
	Then $\mathcal{I}^j_e$ is open in $e + z\mathfrak{g}[[z]]$.
\end{lem}

\begin{proof}
	Let $k \in \mathbb{N}$. It is enough to prove the following
	subsets are open
	\[
	\mathcal{I}^{j,k}_e := \{\gamma \in e + z\mathfrak{g}[[z]] \mid
	\nu_t(F_j(\gamma)) \leq k \}
	\]
	To see this, consider the induced morphism of varieties
	\[
	F_j^k : \mathfrak{g}[[z]]/z^{k+1}\mathfrak{g}[[z]] \rightarrow
	\mathbb{C}[[z]]/z^{k+1}\mathbb{C}[[z]].
	\]
	Let $f_{k+1} : \mathfrak{g}[[z]] \rightarrow \mathfrak{g}[[z]]/z^{k+1}\mathfrak{g}[[z]]$ be the natural map.
	which is surjective. Now we have
	\[
	F_j^k(f_{k+1}(\gamma)) \neq 0 \,\ \iff \nu_t(F_j(\gamma)) \leq k
	\]
	Thus we get
	\[
	\mathcal{I}_e^{j,k} = f_{k+1}^{-1} ((f^{k+1}_0)^{-1}(\{e\}) \cap
	(F^k_j)^{-1}(0)
	\]
	and hence is open.
\end{proof}

We then have following proposition
\begin{prop}\label{prop:2}
	Let $\gamma \in e + z\mathfrak{g}l_m[[z]]$. Then we have
\[
\nu_t(F_j(\gamma)) \geq n(\mu,j)
\]
\end{prop}

\begin{proof}
We can assume with no loss of generality that $e$ is in the jordan form,
with jordan blocks of sizes $m_i$ along the diagonal. Let $\mathfrak{l}$
be the sub-algebra of $gl_m$, consisting of blocks of sizes $m_i$
along the diagonal. We then have
\[
\mathfrak{l} \cong \mathfrak{g}l_{m_1} \times \cdots \times \mathfrak{g}l_{m_r}
\]	
and $e \in \mathfrak{l}$ is regular nilpotent in $\mathfrak{l}$.
Let $e = e_1 +e_2 + \cdots + e_r$ be the jordan decomposition, 
and $\{F_{ij}\}_{j=1}^{m_i}$ be the invariants in $\mathbb{C}[\mathfrak{gl}_{m_i}]^{GL_{m_i}}$ given by the co-efficients
of characteristic polynomials.
For $ M \in \mathfrak{l}$, we think of $M$ as a tuple $(M_1,\ldots,M_r)$,
where $M_i \in \mathfrak{g}l_{m_i}$. Using the formula
\[
det(M - TId) = det(M_1 - TId) \cdots det(M_r - TId)
\]
we can express $F_j \mid_{\mathfrak{l}}$, as a sum of products
of the form $F_{i_1j_1}\cdots F_{i_hj_h}$, where 
$i_1 \neq i_2 \cdots \neq i_h$ and $j_1 + \cdots + j_{h} = j$.
The minimum number of such terms possible is $\mu(j)$ and thus
we get
\[
\nu_t(F_j(\gamma)) \geq n(\mu,j), \,\ \gamma \in e + z\mathfrak{l}[[z]]
\]
since we have $\nu_t(F_{ij}(\gamma)) \geq 1$, for $\gamma \in e + z\mathfrak{l}[[z]]$. 
Now we consider the open set of regular semisimple elements in $e + z\mathfrak{g}[[z]]$ as considered by Khazdan and Lusztig (\cite{KL}[pg~156])
which over $\overline{\mathbb{C}((t))}$ is conjugate to
a matrix of the form $D = diag(a^1_1,\cdots,a^1_{m_1},a^2_1,\cdots,a^r_{1},\cdots,a^r_{m_r})$,
where 
\[
a^i_{\lambda} \in z^{\frac {1}{m_i}}\mathbb{C}[[z^{\frac {1}{m_i}}]]
\]
We will denote this open set by $\mathcal{J}_e$.
Observe we have $D \in \mathfrak{l} \otimes \overline{\mathbb{C}((t))}$ and
$F_{ij}(D) \in z\mathbb{C}[[z]], \forall i,j$.
Choose an $\mathfrak{s}l_2$ triple $\{e,f,h\} \subset \mathfrak{l}$.
Let $V = \mathfrak{z}_{\mathfrak{l}}(f)$. Choose a basis 
$\{v_{ij}\}$
of $V$ consisting of eigenvectors for $t$. Now from a theorem of Kostant(\cite{Kos}),
we can find generators $\{Q_{ij}\}$ for the ring $\mathbb{C}[[\mathfrak{l}]]^L$,
such that we have
\[
Q_{ij}(e + \Sigma a_{pq}v_{pq}) = a_{ij}.
\]
Since we have $\{F_{ij}\}$ are generators for the invariant ring 
$\mathbb{C}[\mathfrak{l}]^L$ as well, we get
\[
Q_{ij}(D) = d_{ij} \in z\mathbb{C}[[z]].
\]
Thus we get
\[
Q_{ij}(D) = Q_{ij} (e + \Sigma d_{pq}v_{pq}).
\]
consequently we get
\[
F_j(\tilde{D}) = F_j(D) = F_j(e + \Sigma d_{pq}v_{pq}),
\]	
where $\tilde{D} \in \mathcal{J}_e$ is conjugate to $D$ over $\overline{\mathbb{C}((t))}$.
In particular, we get
\[
\nu_t(F_j(\tilde{D})) \geq n(\mu,j), \,\ \tilde{D} \in \mathcal{J}_e.
\]
Now to finish off the proof consider the open set $\mathcal{I}^j_e$
as in Lemma~\ref{lem:2}. Since $e \zeta\mathfrak{g}[[z]]$ is an irreducible
contsructible subset of $\mathfrak{g}[[z]]$, we have
\[
\mathcal{J}_e \cap \mathcal{I}^j_e \neq \emptyset.
\]
Thus we get the minimum of $v_t(F_j)$ attained in $e + z\mathfrak{g}[[z]]$
is atleast $\mu(j)$ and hence the lemma is proved.
\end{proof}

\begin{lem}[A combinatorial identity]\label{lem:3}
	We have
	\[
	\Sigma_i i m_i = \frac {1}{2}(\Sigma_j \tilde{m}_j^2 + \tilde{m}_j)
	\]
\end{lem}

\begin{proof}
	Consider the young tableau corresponding to $\mu$.
	Enter the number $i$ on every box in the $i^{th}$ row of the tableau.
	Now summing the resulting set of numbers row-wise, we get
	the number $\Sigma_i im_i$. While summinng up the numbers, column wise,
	we get the number $\frac {1}{2}(\Sigma_j \tilde{m}_j^2 + \tilde{m}_j)$.
	Thus we get the required equality
	\[
	\Sigma_i i m_i = \frac {1}{2}(\Sigma_j \tilde{m}_j^2 + \tilde{m}_j).
	\]
\end{proof}

For $\mathfrak{g} = \mathfrak{s}p_{2n}$, the polynomial functions
$\{F_{2i}\mid_{\mathfrak{s}p_{2n}} \}_{i=1}^n$ is a minimal homogeneous
generating set for the ring $\mathbb{C}[\mathfrak{s}p_{2n}]^{Sp_{2n}}$.
We have the following corollary of the Lemma~\ref{lem:3}
\begin{coro}\label{coro:3}
Let $\mu$ be the partition corresponding to a nilpotent element
$e \in \mathfrak{s}p_{2n} \subset \mathfrak{g}l_{2n}$. Then we have
\[
\Sigma_{j=1}^n n(\mu,2j) = \frac {n}{2} + \frac {1}{2} (dim(Z_{Sp_{2n}}(e)))
\]
\end{coro}
\begin{proof}
	We have
	\[
	n(\mu,2j) = a, \,\ m_1 + \cdots + m_{a-1} < 2j \leq m_1 + \cdots + m_a.
	\]
	Thus we have
	\[
	\lfloor m_1 + \cdots + m_{a-1} \rfloor + 1 \leq j \leq \lfloor m_1 + \cdots + m_a \rfloor.
	\]
	Thus consider the subset of $C_a \subset \{1,\ldots,n\}$, given by
	\[
	C_a = \{j \mid \mu(j) = a \}
	\] 
	We have 
	\[
	\# C_a = \begin{cases}
	\frac {m_a}{2} & if \,\ m_a \,\ \text {is even}\\
	\frac {m_a-1}{2} & if \,\ \text{ $m_a$ is odd and $m_1 + \cdots + m_{a-1}$ is even}\\
	\frac{m_a+1}{2} & if \,\ \text{$m_a$ is odd and $m_1 + \cdots + m_{a-1}$ is
		odd}
	\end{cases}
	\]
	We have now
	\[
	\Sigma_{j=1}^n n(\mu,2j) = \Sigma_{a=1}^r a \# C_a.
	\]
	Now since $\mu$ corresponds to a nilpotent orbt in $\mathfrak{s}p_{2n}$,
	recall we have odd numbers occurs with even multiplicity in $\mu$.
	Now consider the set $C^{o} = \{a \mid \text {$m_a$ is odd}\}$.
	Arrange the numbers in $C^{o}$ in increasing order and let them be
	$\{a_1,\cdots,a_{2q}\}$. We then have
	\[
	a_{i} = a_{i+1} + 1 , \,\ i=1,3,\ldots,2q-1.
	\]
	Thus we get
	that
	\begin{eqnarray*}
	\Sigma_{a \in C^0} a \# C_a & = & \Sigma_{a \in C^o} \frac {am_a}{2} +
	\frac{1}{2}(-a_1 + a_2 - a_3 + \cdots ) \\
	                            & = & \Sigma_{a \in C^o} \frac {am_a}{2} +
	\frac {1}{2} (\frac{1}{2}\#C^o)                            
	\end{eqnarray*}
\end{proof}
Thus we have
\begin{eqnarray*}
\Sigma_{j=1}^n \mu(2j) & = & \Sigma_{a \notin C^0} a\#C_a + \Sigma_{a \in C^o} a\#C_a.\\
                       & = & \frac{1}{2}(\Sigma_a am_a + \frac {1}{2}\#C^o)\\
                       & = & \frac {1}{2} (\frac {1}{2}(\Sigma_k \tilde{m}_k^2 + \#C^o) + n) \\
                       & = & \frac {n}{2} + \frac {1}{2} (dim(Z_{Sp_{2n}}(e)))
\end{eqnarray*}

\begin{rem}
	For $\mathfrak{g} = \mathfrak{s}o_{2n+1} \subset \mathfrak{g}l_{2n+1}$.
	We have $\{F_{2i}\mid_{\mathfrak{g}} \}_{i=1}^{n}$ is a generating set
	for $\mathbb{C}[\mathfrak{g}]^{So_{2n+1}}$. But the analogue of
	Corollary~\ref{coro:3} fails in this case, as can be seen by the following example:\\
	Let $\mathfrak{g} = \mathfrak{s}o(5)$ and $e \in \mathfrak{g}$, whose
	associated partition of $5$ is $\mu = (2,2,1)$. 
	The dual partition
	of $\mu$ is $(3,2)$. 
	We have two invariants
	$F_{2}$ and $F_{4}$ and 
	\[
	n(\mu,2) + n(\mu,4) = 1 + 2 = 3.
	\]
	But on the otherhand, we have
	\[
	\frac {1}{2}(n + dim(Z_{\mathfrak{g}}(e))) = \frac{1}{2}(2 + 3^2 + 2^2 - 1) = 4.
	\]

\end{rem}

\section{Hitchin morphism for symplectic parabolic bundles}
Let $\mathcal{G}_{X,x,\theta}$ be a Bruhat-Tits group scheme over $(X,x)$
for the symplectic group $Sp_{2n}$ and $\theta$ in the interior of the
rational weyl alcove. Let $P_{\theta}$ be the parabolic subgroup of $Sp_{2n}$,
determined by $\theta$. Denote by $\mathfrak{p}_{\theta}$, the lie algebra
of $P_{\theta}$. Let $\mathfrak{n}_{\theta}$ be the nil radical  of
$\mathfrak{p}_{\theta}$.

We then have
\[
\widehat{\mathfrak{p}}_{\theta} := \widehat{Lie(\mathcal{G}_{X,x,\theta})}\mid_x = ev^{-1}(\mathfrak{p}_{\theta})
\]
where we identify the completed local ring $\widehat{\mathcal{O}}_x$
with $\mathbb{C}[[z]]$ and $ev : \mathfrak{s}p_{2n}[[z]] \rightarrow \mathfrak{s}p_{2n}$ is the natural map.
From Lemma we have
\[
\widehat{\mathfrak{p}}_{\theta}^{\vee} = \frac{1}{z}(\mathfrak{n}_{\theta} + z\mathfrak{g}[[z]])
\]
As in the previous section, we denote by $F_j \in \mathbb{C}[\mathfrak{s}p_{2n}]^{Sp_{2n}}$, the invariants defined
by the equation
\[
det(M - TId) = \Sigma_{j=0}^{2n} F_j(M)T^{2n-j}
\]
It is a well-known fact that, we have $F_j = 0, \,\ \text{for $j$ odd}$
and $\{F_{2j}\}_{j=1}^n$ generates the algebra of invariants.
Let $O_{\theta} \in \mathfrak \mathfrak{n}_{\theta}$, be the dense open
$P_{\theta}$ orbit corresponding to the Richardson class. Let
the corresponding partition of $2n$ be $\Lambda = (\lambda_1,\ldots,\lambda_r)$.
\begin{lem}\label{lem:4}
	We have \[
	\nu_t(F_{2j}(\gamma)) \geq -(2j - n(\Lambda,2j)), \,\ \forall \gamma \in \widehat{\mathfrak{p}_{\theta}}^{\vee}.	
	\]
\end{lem}

\begin{proof}
Let $\gamma \in \widehat{\mathfrak{p}}$. We then have
\[
\gamma = \frac {1}{z}(e + z\tilde{\gamma}), \,\ e \in \mathfrak{n}_{\theta}
\]
Thus we have
\[
\nu_z(F_{2j}(\gamma)) = -2j + \nu_z(F_{2j}(e + z \tilde{\gamma}))
\]
Let $\mu = (m_1,\ldots,m_s)$ be the partition of $2n$ corresponding to $e$.
We have from Proposition~\ref{prop:A}, that
\[
\nu_z(F_{2j}(\gamma)) \geq -2j + n(\mu,2j) = -(2j - n(\mu,2j))
\]
Now since we $e \in \overline{O}_{\theta}$, we have
for any $k$,
\[
\Sigma_{i=1}^k m_i \leq \Sigma_{i=1}^k \lambda_i.
\]
Thus we see that
\[
n(\mu,2j) \geq n(\Lambda,2j).
\]
Hence we get
\[
\nu_z(F_{2j}(\gamma)) \geq -(2j-n(\mu,2j)) \geq -(2j - n(\Lambda,2j))
\]
\end{proof}

\begin{coro}
	Let $C >>0$ be such that, we have
	\[
	\nu_z(F_{2j}(\gamma)) \geq -C, \,\ \forall \gamma \in \widehat{\mathfrak{p}}_{\theta}, \,\ j=1,2,\cdots,n.
	\]
	Then image of the Hitchin morphism
	\[
	\chi : T^* Bun_{\mathcal{G}_{X,x,\theta}} \rightarrow \bigoplus 
	\bigoplus_{j=1}^n H^{0}(X,K_X^{2j}(Cx))
	\]
	factors through a subvariety of dimension equal to that of $Bun_{\mathcal{G}_{X,x,\theta}}$.
\end{coro}

\begin{proof}
	From Lemma~\ref{lem:4}, we have $C \leq min\{2j - n(\Lambda,2j)\}_{j=1}^n$
	and the image of $\chi$, lands inside the subspace
	\[
	W  = \bigoplus_{j=1}^{n} H^{0}(X,K_X^{2j}(2j-n(\Lambda,2j)))
	\]
	We have from Riemann-Roch formula,
	\begin{eqnarray*}
	dim(W) & = &  dim(Sp_{2n})(g-1) + \Sigma_{j=1}^n (2j-n(\Lambda,2j))\\
	       & = &  dim(Sp_{2n})(g-1) + \frac{1}{2}(dim(Sp_{2n}) - dim(Z_{Sp_{2n}}(e^o)), \,\ e^o \in O_{\theta}\\
	       & = & dim(Sp_{2n})(g-1) + dim(Sp_{2n}/P_{\theta})\\
	       & = & dim(Bun_{\mathcal{G}_{X,x,\theta}})
	\end{eqnarray*}
\end{proof}

\begin{coro}
	For $G = Sp_{2n}$ and $\theta$ an element in the interior of the rational weyl alcove, we have\\
	(i)$dim(\mathcal{N}ilp_{X,x,\theta}) = dim (Bun_{\mathcal{G}_{X,x,\theta}})$\\
	(ii)The Hitchin morphism $\chi$ is flat. Further every irreducible
	component of the fiber of $\chi$ has dimension equal to that of
	$Bun_{\mathcal{G}_{X,x,\theta}}$.\\
	(iii)$T^*(Bun_{\mathcal{G}_{X,x,\theta}})$ is a local complete intersection
	and $\mathcal{N}ilp_{X,x,\theta}$ is a Lagrangian complete intersection
	in $T^*(Bun_{\mathcal{G}_{X,x,\theta}})$.	
\end{coro}

\begin{proof}
	Proof follows exactly as in the proof of Corollary~\ref{coro:1}
\end{proof}

\begin{center}
	{ACKNOWLEDGEMENT}
\end{center}

This work was done while the author was supported by a post-doctoral fellowship
at the Tata Institute of Fundamental Research, Mumbai.

\bibliographystyle{amsplain}
\bibliography{p1}

\end{document}